\documentclass[12pt]{amsart}

\usepackage[breaklinks]{hyperref}
\usepackage{amsmath}
\usepackage{amssymb}
\usepackage{mathrsfs}
\usepackage{amsthm}
\usepackage[utf8]{inputenc}
\usepackage{svninfo}
\usepackage{prettyref}
\ifx\FrenchText\undefined
\usepackage[british]{babel}
\else
\usepackage[french]{babel}
\AtBeginDocument{%
\catcode`\:=12%
\catcode`\!=12%
}
\fi

\makeatletter

\ifx\Presentation\undefined
\fi

\newcommand{\fref}[1]{\prettyref{#1}}
\newrefformat{item}{\ref{#1}}
\newrefformat{apx}{Appendix~\ref{#1}}
\newrefformat{cha}{Chapter~\ref{#1}}
\newrefformat{sec}{Section~\ref{#1}}

\ifx\thmnum\undefined
\newtheorem{thmnum}{}[section]
\fi

\newcommand{\mynewthm}[3][]{%
  \def\PARAM{#1}
  \ifx\PARAM\empty
  \newtheorem{#2}[thmnum]{#3}
  \else
  \newtheorem{#2}{#3}[#1]
  \fi
  \newtheorem*{#2*}{#3}%
  \newrefformat{#2}{#3~\ref{##1}}%
}

\ifx\FrenchText\undefined
\newcommand{\ThmLabel}{Theorem}
\newcommand{\PrpLabel}{Proposition}
\newcommand{\LemLabel}{Lemma}
\newcommand{\FctLabel}{Fact}
\newcommand{\CorLabel}{Corollary}
\newcommand{\DfnLabel}{Definition}
\newcommand{\ConvLabel}{Convention}
\newcommand{\NtnLabel}{Notation}
\newcommand{\CstLabel}{Construction}
\newcommand{\ExmLabel}{Example}
\newcommand{\RmkLabel}{Remark}
\newcommand{\QstLabel}{Question}

\else
\newcommand{\ThmLabel}{\iflanguage{french}{Théorème}{Theorem}}
\newcommand{\PrpLabel}{Proposition}
\newcommand{\LemLabel}{\iflanguage{french}{Lemme}{Lemma}}
\newcommand{\FctLabel}{\iflanguage{french}{Fait}{Fact}}
\newcommand{\CorLabel}{\iflanguage{french}{Corollaire}{Corollary}}
\newcommand{\DfnLabel}{\iflanguage{french}{Définition}{Definition}}
\newcommand{\ConvLabel}{Convention}
\newcommand{\NtnLabel}{Notation}
\newcommand{\CstLabel}{Construction}
\newcommand{\ExmLabel}{\iflanguage{french}{Exemple}{Example}}
\newcommand{\RmkLabel}{\iflanguage{french}{Remarque}{Remark}}
\newcommand{\QstLabel}{Question}

\fi

\theoremstyle{plain}
\mynewthm{thm}{\ThmLabel}
\mynewthm{prp}{\PrpLabel}
\mynewthm{lem}{\LemLabel}
\mynewthm{fct}{\FctLabel}
\mynewthm{cor}{\CorLabel}

\theoremstyle{definition}
\mynewthm{dfn}{\DfnLabel}
\mynewthm{conv}{\ConvLabel}
\mynewthm{conj}{Conjecture}
\mynewthm{ntn}{\NtnLabel}
\mynewthm{cst}{\CstLabel}

\theoremstyle{remark}
\mynewthm{rmk}{\RmkLabel}
\mynewthm{qst}{\QstLabel}
\mynewthm{exm}{\ExmLabel}

\ifx\Presentation\undefined

\newcommand{\myenumlabel}[1]{\textnormal{(\roman{#1})}}
\renewcommand{\theenumi}{\myenumlabel{enumi}}
\fi

\newcounter{cycprfcnt}
{\begin{list}{\PackageWarning{begnac}{Label required for cycprf}}%
  {%
    \setcounter{cycprfcnt}{1}
    \setlength{\itemindent}{0.5\leftmargin}%
    \setlength{\leftmargin}{0pt}%
    \newcommand{\cpcurr}{\myenumlabel{cycprfcnt}}%
    \newcommand{\cpnext}{\addtocounter{cycprfcnt}{1}\cpcurr}%
    \newcommand{\cpprev}{\addtocounter{cycprfcnt}{-1}\cpcurr}%
    \newcommand{\cpnum}[1]{\setcounter{cycprfcnt}{##1}\cpcurr}%
    \newcommand{\cpfirst}{\cpnum{1}}%
  }%
}%
{\qedhere\end{list}}%

\def\indsym#1#2{%
  \setbox0=\hbox{$\m@th#1x$}%
  \kern\wd0%
  \hbox to 0pt{\hss$\m@th#1\mid$\hbox to 0pt{$\m@th#1^{#2}$\hss}\hss}%
  \lower.9\ht0\hbox to 0pt{\hss$\m@th#1\smile$\hss}%
  \kern\wd0}
\newcommand{\ind}[1][]{\mathop{\mathpalette\indsym{#1}}}

\def\nindsym#1#2{%
  \setbox0=\hbox{$\m@th#1x$}%
  \kern\wd0%
  \hbox to 0pt{\hss$\m@th#1\not$\kern1.4\wd0\hss}
  \hbox to 0pt{\hss$\m@th#1\mid$\hbox to 0pt{$\m@th#1^{#2}$\hss}\hss}%
  \lower.9\ht0\hbox to 0pt{\hss$\m@th#1\smile$\hss}%
  \kern\wd0}

\def\dotminussym#1#2{%
  \setbox0=\hbox{$\m@th#1-$}%
  \kern.5\wd0%
  \hbox to 0pt{\hss\hbox{$\m@th#1-$}\hss}%
  \raise.6\ht0\hbox to 0pt{\hss$\m@th#1.$\hss}%
  \kern.5\wd0}

\renewcommand{\emptyset}{\varnothing}
\renewcommand{\setminus}{\smallsetminus}
\def\models{\vDash}

\newcommand{\concat}{{^\frown}}
\newcommand{\rest}{{\restriction}}

\newcommand{\sfrac}[2]{\hbox{$\frac{#1}{#2}$}}
\newcommand{\half}[1][1]{\sfrac{#1}{2}}

\DeclareMathOperator{\tp}{tp}

\newcommand{\Cb}{\mathrm{Cb}}

\DeclareMathOperator{\tS}{S}
\DeclareMathOperator{\dcl}{dcl}
\DeclareMathOperator{\acl}{acl}

\DeclareMathOperator{\id}{id}

\DeclareMathOperator{\Aut}{Aut}

\DeclareMathOperator{\Pert}{Pert}

\newcommand{\fB}{\mathfrak{B}}

\newcommand{\cA}{\mathcal{A}}
\newcommand{\cB}{\mathcal{B}}

\newcommand{\cH}{\mathcal{H}}

\newcommand{\cL}{\mathcal{L}}

\newcommand{\cU}{\mathcal{U}}

\newcommand{\sA}{\mathscr{A}}
\newcommand{\sB}{\mathscr{B}}
\newcommand{\sC}{\mathscr{C}}

\newcommand{\sL}{\mathscr{L}}

\newcommand{\bC}{\mathbb{C}}

\newcommand{\bE}{\mathbb{E}}

\newcommand{\bN}{\mathbb{N}}
\newcommand{\bP}{\mathbb{P}}
\newcommand{\bQ}{\mathbb{Q}}
\newcommand{\bR}{\mathbb{R}}
\newcommand{\bZ}{\mathbb{Z}}

\makeatother

\begin{document}

\title[On perturbations of a generic automorphism]{On perturbations of
  Hilbert spaces and probability algebras with a generic automorphism}

\author{Itaï \textsc{Ben Yaacov}}
\address{Itaï \textsc{Ben Yaacov} \\
  Université Claude Bernard -- Lyon 1 \\
  Institut Camille Jordan \\
  43 boulevard du 11 novembre 1918 \\
  69622 Villeurbanne Cedex \\
  France}

\urladdr{\url{http://math.univ-lyon1.fr/~begnac}}

\author{Alexander Berenstein}
\address{Alexander Berenstein
  Universidad de los Andes \\
  Departamento de Matemáticas \\
  Carrera 1 \# 18A-10, Bogotá, Colombia ; and
  \newline\indent
  Université Claude Bernard -- Lyon 1 \\
  Institut Camille Jordan \\
  43 boulevard du 11 novembre 1918 \\
  69622 Villeurbanne Cedex \\
  France}
\email{\url{aberenst@uniandes.edu.co}}
\urladdr{\url{http://matematicas.uniandes.edu.co/~aberenst/}}

\thanks{Research supported by NSF grant DMS-0500172,
  by ANR chaire d'excellence junior THEMODMET (ANR-06-CEXC-007)
  and by Marie Curie research network ModNet}

\svnInfo $Id: GenAutPert.tex 922 2009-06-02 13:14:12Z begnac $
\thanks{\textit{Revision} {\svnInfoRevision} \textit{of} \today}

\begin{abstract}
  We prove that $IHS_A$, the theory of infinite dimensional
  Hilbert spaces equipped with a
  generic automorphism, is $\aleph_0$-stable up to perturbation of the
  automorphism, and admits prime models up to perturbation over any set.
  Similarly, $APr_A$, the theory of atomless probability algebras
  equipped with a generic automorphism is $\aleph_0$-stable up to
  perturbation.
  However, not allowing perturbation it is not even superstable.
\end{abstract}

\maketitle

\section*{Introduction}
It was proved by Chatzidakis and Pillay 
\cite{Chatzidakis-Pillay:GenericStructuresAndSimpleTheories} that if $T$ is a
first order superstable theory, and the theory
$T_\tau = T \cup \{\tau \text{ is an automorphism}\}$
has a model companion $T_A$, then $T_A$ is supersimple.
Throughout this paper we refer to $T_A$ (when it exists) as the theory
of models of $T$ equipped with a \emph{generic} automorphism.

Continuous first order logic is an extension of first order logic,
introduced in \cite{BenYaacov-Usvyatsov:CFO} as a formalism for a
model theoretic treatment of metric structures
(see also \cite{BenYaacov-Berenstein-Henson-Usvyatsov:NewtonMS} for a
general exposition of the model theory of metric structures).
It is a natural question to ask whether the theorem of
Chatzidakis and Pillay generalises to continuous logic and metric
structures.

The proof of Chatzidakis and Pillay would hold in metric
structures if we used the classical definitions of superstability and
supersimplicity literally (namely, types do not fork over finite
sets).
These definitions, however, are known to be too strong in metric
structures, and need to be weakened somewhat in order to make sense.
For example, the theory of Hilbert spaces has a countable language, is
totally categorical and does not satisfy the classical definition of
superstability.

The standard definition for $\aleph_0$-stability and superstability
for metric structures
(\cite{Iovino:StableBanach}, and later \cite{BenYaacov:Morley})
comes from measuring the size of a type space
not by its cardinality but by its density character in the metric
induced on it from the structures.
A continuous theory is supersimple if for
every $\varepsilon > 0$, the $\varepsilon$-neighbourhood of a type
does not fork over a finite set of parameters,
or equivalently, if ordinal Lascar ranks
corresponding to  ``$\varepsilon$-dividing'' exist.
A theory is superstable if and only if it is stable and supersimple.
Similarly, $\aleph_0$-stability is equivalent to the existence of ordinal
$\varepsilon$-Morley ranks, which may be defined via a metric variant of the
classical Cantor-Bendixson ranks
(see \cite{BenYaacov:TopometricSpacesAndPerturbations}
for a general study of such ranks).

With these corrected definitions,
the class of $\aleph_0$-stable theories is rich with
examples: Hilbert spaces, probability algebras, $L^p$ Banach lattices
and so on.
Furthermore, many classical results can be generalised.
For example, an $\aleph_0$-stable theory has prime models
over every set, an uncountably categorical theory in a countable
language is $\aleph_0$-stable, and so on (see
\cite{BenYaacov:Morley}).
A somewhat more involved preservation result was shown by the first
author \cite{BenYaacov:SuperSimpleLovelyPairs}, namely that the theory
of lovely pairs of models of a supersimple (respectively, superstable)
theory is again supersimple (respectively, superstable).

With superstability and supersimplicity defined as above,
the question whether a superstable theory with a generic automorphism
is supersimple arises again.
A specific instance of this question was asked by the second author
and C.\ Ward Henson
\cite{Berenstein-Henson:ProbabilityWithAutomorphism} regarding the
theory of probability algebras equipped with a generic automorphism.
It was answered negatively by the first author, showing that
probability algebras with a generic automorphism are not
superstable.
The proof appears in \fref{sec:ProbAutNotSuperstable}.

However, the notions of $\aleph_0$-stability and/or superstability mentioned
above might still be too strong: while they consider types of tuples
up to arbitrarily small distance, one may further relax this and
consider types also up to arbitrarily small perturbations of the
entire language, or parts thereof.
This idea can be formalised with the theory of perturbations as
developed in \cite{BenYaacov:Perturbations} and somewhat restated
in \cite[Section~4]{BenYaacov:TopometricSpacesAndPerturbations}.
We shall assume some familiarity with the second reference.

The goal of this paper is to study carefully two examples: Hilbert
spaces and probability algebras, both equipped with a generic
automorphism.
The theory $APr$ of atomless probability algebras
and the theory $IHS$ of infinite dimensional Hilbert spaces have some
features in common.
They are $\aleph_0$-stable, separably categorical over any finite set
of parameters and types over sets are stationary.
It follows (see \cite{Berenstein-Henson:ProbabilityWithAutomorphism})
that both $IHS_\tau$ and $APr_\tau$ admit model companions
$IHS_A$ and $APr_A$.

In \fref{sec:Hilbert} we deal with the theory $IHS_A$ of Hilbert
spaces equipped with a generic automorphism.
We recall some if its properties from
\cite{BenYaacov-Usvyatsov-Zadka:HilbertWithAutomorphism}.
We use a Corollary of the Weyl-von Neumann-Berg Theorem to show that
$IHS_A$ is $\aleph_0$-stable up to perturbation (of the automorphism),
and admits prime models up to perturbation over any set.
Unlike the arguments in
\cite{BenYaacov-Usvyatsov-Zadka:HilbertWithAutomorphism},
our arguments can be extended to a generic action of a finitely
generated group of automorphisms (i.e., a generic unitary
representation, see \cite{Berenstein:HilbertWithAutomorphismGroups})
and even to Hilbert spaces equipped with a
generic action of a fixed finitely generated $C^*$-algebra.
This section also serves as a soft analogue for the main results of the
other sections.

In \fref{sec:Probability} we deal with the theory $APr_A$ of
probability algebras with a generic automorphism, first studied in
\cite{Berenstein-Henson:ProbabilityWithAutomorphism}.
Specifically, we show that $APr_A$ is $\aleph_0$-stable
up to perturbations of the automorphism.
It is an open question if $APr_A$ admits prime models up to
perturbations.

In \fref{sec:ProbAutNotSuperstable} we conclude with
the first author's proof that without perturbation the theory $APr_A$
is not superstable, showing that the results of
\fref{sec:Probability} are in some sense optimal.

\section{Hilbert spaces with an automorphism}
\label{sec:Hilbert}

Let us consider a Hilbert space $\cH$ and let
$B(\cH)$ denote the space of bounded linear operators on $\cH$.
We recall that the \emph{operator norm} of $T \in B(\cH)$ is
$\|T\|=\sup_{\|x\|=1}\|T(x)\|$.
We also recall the notions of the \emph{spectrum},
\emph{punctual spectrum} and \emph{essential spectrum} of
an operator $T \in B(\cH)$:
\begin{align*}
  \sigma(T) & =
  \{\lambda \in \bC\colon T-\lambda I \text{ is not invertible}\}, \\
  \sigma_p(T) & =
  \{\lambda \in \bC\colon \ker(T-\lambda I) \neq 0\}, \\
  \sigma_e(T) & =
  \{\text{non isolated points of } \sigma(T)\} \cup
  \{\lambda \in \bC\colon \dim \ker (T-\lambda I) = \infty\}.
\end{align*}

\begin{dfn}
  Let $\cH$ be a Hilbert space, $T_0,T_1 \in B(\cH)$.
  We say that $T_0$ and $T_1$ are
  \emph{approximately unitarily equivalent} if there is a sequence of
  unitary operators $\{U_n\}_{n \in \bN}$ such that
  $\|T_0-U_nT_1U_n^*\| \to 0$.
\end{dfn}

\begin{fct}[Weyl-von Neumann-Berg Theorem {\cite[p.\ 60]{Davidson:CsAlgebrasByExample}}]
  \label{fct:WvNB}
  Let $\cH$ be a Hilbert space and let $T_0,T_1 \in B(\cH)$ be normal
  operators.
  Then $T_0$ and $T_1$ are approximately unitarily equivalent if and only if
  \begin{enumerate}
  \item $\sigma_e(T_0)=\sigma_e(T_1)$
  \item $\dim \ker(T_0-\lambda I)=\dim \ker(T_1-\lambda I)$
    for all $\lambda$ in $\mathbb{C}\setminus \sigma_e(T_0)$.
  \end{enumerate}
\end{fct}

When considering a Hilbert space as a continuous structure we shall
replace it with its unit ball, as described in
\cite{BenYaacov:NakanoSpaces}.
We shall use the language
$\cL = \{0,-,\dot 2,\half[x+y]\}$,
where $\dot 2x = \min(2,\frac{1}{\|x\|})x$
and $d(x,y) = \| \half[x-y] \|$.
Notice that we can recover the norm as
$\|x\| = d(x,-x)$.
An axiomatisation for the class of (unit balls of) Banach spaces in
this language, excluding the symbol $\dot 2$,
appears in \cite{BenYaacov:NakanoSpaces}.
The symbol $\dot 2$ serves as a Skolem function for the fullness axiom
there, yielding a universal theory.
A Banach space is a Hilbert space if and only if the parallelogram
identity holds, which is a universal condition as well:
\begin{gather*}
  \|x+y\|^2+\|x-y\|^2=2\|x\|^2+2\|y\|^2
\end{gather*}
We obtain that the class of Hilbert spaces is elementary, admitting a
universal theory $HS$.
Its model companion is $IHS$, the theory of infinite dimensional
Hilbert spaces, obtained by adding the appropriate scheme of
existential conditions.
It is easy to check that the theory $HS$ has the amalgamation
property, so $IHS$ eliminates quantifiers (i.e., it is the model
completion of $HS$).

Now let $\tau$ be a new unary function symbol and
let $\cL_\tau$ be $\cL\cup\{\tau\}$.
Let $IHS_\tau$ be the theory
$IHS \cup \{\tau \text{ is an automorphism}\}$.
Since $IHS$ is $\aleph_0$-stable and separably categorical even
after naming finitely many constants,
the theory $IHS_\tau$ admits a model companion $IHS_A$
(see \cite{Berenstein-Henson:ProbabilityWithAutomorphism}).
The universal part of $IHS_\tau$ is
$(IHS_\tau)^\forall = (HS_\tau)^\forall
= HS \cup \{\tau \text{ is a linear and isometric}\}$.
It is again relatively easy to check that
$(HS_\tau)^\forall$ has the amalgamation property.
Indeed,
if $(\cH_0,\tau_0) \subseteq (\cH_i,\tau_i)$ for $i=1,2$
then we may write
$(\cH_i,\tau_i) = (\cH_0,\tau_0) \oplus (\cH_i',\tau_i')$,
where $\oplus$ is the orthogonal direct sum, and then
$(\cH_0,\tau_0) \oplus (\cH_1',\tau_1') \oplus (\cH_2',\tau_2')$
will do.
It thus follows that $IHS_A$ eliminates quantifiers as well.

\begin{prp}[Ben Yaacov, Usvyatsov, Zadka
  {\cite{BenYaacov-Usvyatsov-Zadka:HilbertWithAutomorphism}}]
  \label{prp:HilbertAutSpectrum}
  Let $\cH$ be a separable Hilbert space and let $\tau$ be a unitary
  operator on $\cH$.
  Then $(\cH,\tau) \models IHS_A$
  (i.e., $(\cH,\tau)$ is existentially closed as a model of $IHS_\tau$)
  if and only if $\sigma(\tau)=S^1$.
\end{prp}
\begin{proof}
  Clearly, if $(\cH,\tau)$ is existentially closed, then
  $\sigma(\tau)=S^1$.
  On the other hand, assume that $(\cH,\tau) \models IHS_\tau$ and that
  $\sigma(\tau)=S^1$.
  Passing to an elementary substructure, we may assume that $\cH$ is
  separable.
  Now let $(\cH_0,\tau_0)$ be separable and existentially closed.
  Since $IHS$ is separably categorical, we may assume that $\cH_0=\cH$.
  Since $\sigma(\tau_0)=\sigma(\tau)=S^1$, by \fref{fct:WvNB} there is
  a sequence $\{U_n\}_{n\in \omega}$ of unitary operators on $\cH$
  such that $U_n\tau_1U_n^* \to \tau_0$ in norm.
  It follows that if $\cU$ is a non-principal ultra-filter on $\bN$
  then $\Pi_\cU (\cH,U_n\tau  U_n^*) = \Pi_\cU (\cH,\tau_0)$.

  On the other hand, $(\cH,U_n\tau U_n^*)\cong (\cH, \tau)$
  for all $n \in \bN$.
  Thus
  $\Pi_{\mathcal{U}}(\cH,\tau)\cong \Pi_{\mathcal{U}}(\cH,\tau_0)$,
  whereby $(\cH,\tau) \equiv (\cH,\tau_0) \models IHS_A$.
\end{proof}

\begin{rmk}
  Henson and Iovino observed that the
  theory $IHS_A$ is not $\aleph_0$-stable (or even small) in the
  sense defined in the introduction.
  Indeed, let $(\cH,\tau) \models IHS_A$ be $\aleph_1$-saturated and for
  each $\lambda \in S^1$ let $v_\lambda\in H$ be a normal vector such
  that $\tau v_\lambda = \lambda v_\lambda$.
  Then $d\bigl(\tp(v_\lambda),\tp(v_\rho)\bigr) = \sqrt{2}$ for
  $\lambda \neq \rho$.
  Thus the metric density character of $\tS_1(\emptyset)$ is the
  continuum.

  On the other hand, it is shown in
  \cite{BenYaacov-Usvyatsov-Zadka:HilbertWithAutomorphism} that $IHS_A$
  is superstable.
\end{rmk}

Let $\dcl_\tau$ and $\acl_\tau$ denote the definable and algebraic
closure (in the real sort) in models of $IHS_A$.
We claim that if  $(\cH,\tau)\models IHS_A$ and $A\subseteq H$, then
$\dcl_\tau(A) = \acl_\tau(A)
= \dcl\left( \bigcup_{n \in \bZ}\tau^n(A) \right)$,
where $\dcl(A)$ is the definable closure of $A$ in the language $\cL$.
Indeed, let
$B = \dcl\left( \bigcup_{n \in \bZ}\tau^n(A) \right)$.
Then clearly $B \subseteq \dcl_\tau(A)$.
On the other hand, we may decompose
$(\cH,\tau) = (B,\tau\rest_B) \oplus (B',\tau\rest_{B'})$,
in which case
$(\cH,\tau) \preceq
(B,\tau\rest_B) \oplus \bigoplus_{n\in\bN} (B',\tau\rest_{B'})$,
showing that $\acl_\tau(A) \subseteq B$.

We may similarly characterise non forking in models of $IHS_A$.
For $(\cH,\tau) \models IHS_A$ and subsets $A,B,C\subseteq \cH$,
say that  $A\ind_B C$ if $P_{\dcl_\tau(B)}(a)=P_{\dcl_\tau(BC)}(a)$
for every $a\in A$.
We leave it to the reader to check that
$\ind$ satisfies the usual axioms of a stable notion of independence
(invariance, symmetry, transitivity, and so on), and therefore
coincides with non forking.

\begin{prp}
  Let $(\cH,\tau)\models IHS_A$, $A, B \subseteq  \cH$.
  Then $\tp(A/B)$ is stationary and $\Cb(A/B)$ is inter-definable
  with the set $C = \{P_{\dcl_\tau(B)}(a)\})_{a\in A}$.
\end{prp}
\begin{proof}
  Stationarity follows from the characterisation of independence (and
  from quantifier elimination).
  It is also clear that $C \subseteq \dcl_\tau(B)$ and
  $A \ind_C B$, and since we already know that $\tp(A/C)$ is
  stationary as well, we obtain $\Cb(A/B) \subseteq C$.

  For the converse it suffices to show that for every $a \in A$,
  the projection
  $P_{\dcl_\tau(B)}(a)$
  belongs to the definable closure of any Morley sequence in
  $\tp(A/B)$.
  So let $(A_n)_{n\in \bN}$ be such a Morley sequence.
  Then $P_{\dcl_\tau(B)}(a_n) = P_{\dcl_\tau(B)}(a)$
  for all $a \in A$  and all $n$, so
  $\{a_n-P_{\dcl_\tau(B)}(a)\}_{n \in \bN}$
  forms an orthogonal sequence of bounded norm.
  Thus
  \begin{gather*}
    \sum_{n = 0}^{m-1} \frac{a_n}{m}
    =
    P_{\dcl_\tau(B)}(a)
    + \sum_{n=0}^{m-1}\frac{a_n-P_{\dcl_\tau(B)}(a)}{m} \to
    P_{\dcl_\tau(B)}(a).
    \qedhere
  \end{gather*}
\end{proof}

It follows that $IHS_A$ has weak elimination of imaginaries, namely that
for every imaginary element $e$ there exists a real tuple
(possibly infinite) $A$ such that
$A \subseteq \acl_\tau(e)$,
$e \in \dcl^{eq}(A)$.

We now turn to perturbations of the automorphism in models of $IHS_A$.
Let $(\cH_i,\tau_i) \models IHS_A$ for $i = 0,1$, and let $r \geq 0$.
We define an \emph{$r$-perturbation} of $(\cH_0,\tau_0)$
to $(\cH_1,\tau_1)$ to be an isometric isomorphism of Hilbert spaces
$U\colon \cH_0 \cong \cH_1$ which satisfies in addition
\begin{gather*}
  \| U \tau_0 U^{-1} - \tau_1 \| \leq r.
\end{gather*}
The set of all $r$-perturbations will be denoted
$\Pert_r\bigl( (\cH_0,\tau_0), (\cH_1,\tau_1) \bigr)$.
It is fairly immediate to verify that this indeed satisfies all the
conditions stated in
\cite[Theorem~4.4]{BenYaacov:TopometricSpacesAndPerturbations},
and therefore does indeed correspond to a perturbation system as
defined there.

\begin{lem}
  \label{lem:SpectralInclusion}
  Let $(\cH_0,\tau_0) \subseteq (\cH_i,\tau_i)$ be separable models of
  $IHS_\tau$ for $i = 1,2$.
  Then we may write
  $(\cH_i,\tau_i) = (\cH_0,\tau_0) \oplus (\cH_i',\tau_i')$,
  and let us assume that $\sigma(\tau_1') \subseteq \sigma(\tau_2')$
  and that $\sigma(\tau_1')$ has no isolated points.

  Then for every $\varepsilon > 0$ there is an isometric isomorphism
  $U\colon \cH_1 \oplus \cH_2' \cong \cH_2$, which fixes $\cH_0$
  such that
  $\| U(\tau_1\oplus\tau_2')U^{-1} - \tau_2 \| \leq \varepsilon$.
\end{lem}
\begin{proof}
  Under the assumptions we have
  $\sigma(\tau_1'\oplus\tau_2') = \sigma(\tau_2')$.
  We also assume that $\sigma(\tau_1')$ has no isolated points.
  Therefore, if $\lambda \in \sigma(\tau_1'\oplus\tau_2')$ is isolated
  then its eigenspace in $\cH_1' \oplus \cH_2'$ is  entirely
  contained in $\cH_2'$, so the multiplicity (possibly infinite)
  of $\lambda$ is the same for
  $\tau_1'\oplus\tau_2'$ and for $\tau_2'$.
  It follows that the hypotheses of \fref{fct:WvNB} hold,
  and we obtain
  $V\colon \cH_1' \oplus \cH_2' \cong \cH_2'$
  such that 
  $\| V(\tau_1'\oplus\tau_2')V^{-1} - \tau_2' \| \leq \varepsilon$.
  Then $U = \id_{\cH_0} \oplus V$ will do.
\end{proof}

\begin{thm}
  \label{thm:HilbertAutAleph0Stable}
  The theory $IHS_A$ is $\aleph_0$-stable up to perturbation of the
  automorphism.
\end{thm}
\begin{proof}
  Let $(\cH_0,\tau_0), (\cH_1',\tau_1')\models IHS_A$ be separable, and
  let
  $(\cH_1,\tau_1) = (\cH_0,\tau_0) \oplus (\cH_1',\tau_1')$.
  By \fref{prp:HilbertAutSpectrum} we have
  $(\cH_1,\tau_1) \models IHS_A$,
  so $(\cH_0,\tau_0) \preceq (\cH_1,\tau_1)$
  by model completeness.
  It will therefore be enough to show that every type
  over $\cH_0$ is realised, up to perturbation, in
  $(\cH_1,\tau_1)$.
  Such a type can always be realised in a separable elementary
  extension $(\cH_2,\tau_2) \succeq (\cH_1,\tau_1)$.
  Then $(\cH_0,\tau_0) \subseteq (\cH_2,\tau_2)$
  and we may decompose the latter as
  $(\cH_2,\tau_2) = (\cH_0,\tau_0) \oplus (\cH_2',\tau_2')$.

  Notice that $(\cH_1',\tau_1') \subseteq (\cH_2',\tau_2')$, so
  $\sigma(\tau_1')=\sigma(\tau_2')=S^1$.
  We may therefore apply \fref{lem:SpectralInclusion},
  obtaining for every $\varepsilon>0$ there an isometric
  isomorphism $U_\varepsilon\colon \cH_1\rightarrow \cH_2$
  fixing $\cH_0$ such that
  $\|\tau_1 - U_\varepsilon^{-1}\tau_2 U_\varepsilon\|
  < \varepsilon$.

  We have thus shown that every type over $\cH_0$ is realised, up to
  arbitrarily small perturbation of the automorphism, in a fixed
  separable extension $(\cH_1,\tau_1) \succeq (\cH_0,\tau_0)$,
  as desired.
\end{proof}

\begin{rmk}
  Let $G$ be a finitely generated discrete group and let $IHS_{gG}$
  be the theory of Hilbert spaces with a generic action of $G$ by
  automorphism (see \cite{Berenstein:HilbertWithAutomorphismGroups}).
  Using Voiculescu's Theorem \cite{Davidson:CsAlgebrasByExample}
  in place of \fref{fct:WvNB},
  the same argument shows that the theory $IHS_{gG}$ is
  $\aleph_0$-stable up to perturbations of the automorphisms.
  This can even be further extended to the theory $IHS_{g\cA}$ of
  a generic presentation of a finitely generated $C^*$-algebra $\cA$.
\end{rmk}

\begin{prp}
  The theory $IHS_A$ has prime models up to perturbation
  over sets (of real or imaginary elements).

  By this we mean that for every set $A$ in a model of $IHS_A$ there
  exists a model $(\cH_1,\tau_1)$, containing $A$,
  such that if $(\cH_2,\tau_2)$ is any other model which contains $A$
  then, up to arbitrarily small perturbation of $\tau_2$ to $\rho_2$,
  we can embed $(\cH_1,\tau_1)$ elementarily in
  $(\cH_2,\rho_2)$ over $A$.
\end{prp}
\begin{proof}
  We may assume that the set $A$ over which we seek a prime model is
  algebraically closed.
  By weak elimination of imaginaries we may assume that $A$ is a real
  set, and we may further assume that $A = \dcl_\tau(A)$.
  It is therefore a Hilbert subspace $\cH_0$ (possibly finite
  dimensional) on which $\tau_0 = \tau\rest_{\cH_0}$ is an
  automorphism.
  Moreover, since $IHS_A$ eliminates quantifiers, the type of
  $\cH_0$ is determined by the pair $(\cH_0,\tau_0)$, and there is no
  need to consider the ambient structure.

  If $(H_0,\tau_0)\models IHS_A$ there is nothing to prove.
  Otherwise $\sigma(\tau_0) \subsetneq S^1$.
  Let $(H_1',\tau_1')\models IHS_\tau$ be separable such that
  $\sigma(\tau_1') = \overline{S^1 \setminus \sigma(\tau_0)}$
  (for example we may take $H_1' = L_2(S^1 \setminus \sigma(\tau_0))$
  in the Lebesgue measure, with
  $(\tau_1'f)(x) = xf(x)$).
  Let
  $(\cH_1,\tau_1) = (\cH_0,\tau_0) \oplus (\cH_1',\tau_1')$.
  Clearly $\sigma(\tau_1) = S^1$, so
  $(\cH_1,\tau_1) \models IHS_A$,
  and we shall prove that it is prime, up to perturbation of the
  automorphism, over $\cH_0$.

  So let $(\cH_0,\tau_0) \subseteq (\cH_2,\tau_2) \models IHS_A$ and we
  may assume that $\cH_2$ is separable.
  As usual, we may decompose it as
  $(\cH_2,\tau_2) = (\cH_0,\tau_0) \oplus (\cH_2',\tau_2')$.
  Since $\sigma(\tau_2) = S^1$, we necessarily have
  $\sigma(\tau_2') \supseteq S^1 \setminus \sigma(\tau_0)$,
  and since $\sigma(\tau_2')$ is moreover closed, it contains
  $\sigma(\tau_1)$.
  By \fref{lem:SpectralInclusion}, for every $\varepsilon > 0$ there
  exists an isometric isomorphism
  $U\colon \cH_1 \oplus \cH_2' \cong \cH_2$ fixing $\cH_0$
  such that
  $\| U(\tau_1\oplus\tau_2')U^{-1} - \tau_2 \| \leq \varepsilon$.
  By \fref{prp:HilbertAutSpectrum} we also have
  $(\cH_1,\tau_1) \preceq (\cH_1,\tau_1) \oplus (\cH_2',\tau_2')$

  Thus $\rho_2 = U(\tau_1\oplus\tau_2')U^{-1}$ is as desired.
\end{proof}

\section{Probability algebras with an automorphism}
\label{sec:Probability}

By a probability space we mean a triplet $(X,\cB,\mu)$,
where $X$ is a set, $\cB$ a $\sigma$-algebra of subsets of $X$, $\mu$ a $\sigma$-additive
positive measure on $\cB$ such that $\mu(X)=1$.
A probability space  $(X,\cB,\mu)$ is
called \emph{atomless} if for every $A\in \cB$ there is $C\in \cB$ such that
$\mu(A \cap C)=\frac{1}{2}\mu(A)$.
We say that two elements $A,B\in \cB$
determine the same \emph{event}, and write $A \sim_\mu B$ if
$\mu(A\triangle B)=0$.
The relation $\sim_\mu$ is an equivalence relation and the
collection of classes is denoted by $\overline B$ and is called the
\emph{measure algebra} associated to $(X,\cB,\mu)$.
Operations such as
unions, interesections and complements are well defined for events, as
well as the measure.
The distance between two events $a,b\in \overline B$ is
given by the measure of their symmetric difference.
This renders $\overline \cB$ a complete metric space.

Conversely, let
$(\sB,0,1,\cdot^c,\cup,\cap)$ be a Boolean algebra
and assume that $d$ is a complete metric on $\sB$.
Let $\mu(x)$ be an abbreviation for $d(0,x)$ and assume furthermore
that $d(x,y) = \mu(x \triangle y)$,
$\mu(x) + \mu(y) = \mu(x\cap y) + \mu(x \cup y)$
and $\mu(1) = 1$.
Then $\sB$ is the probability algebra associated to some probability
space
(and we may moreover take that space to be the Stone space of $\sB$,
equipped with the Borel $\sigma$-algebra).

We may view probability algebras as continuous
structures in the language
$\cL_{Pr} = \{0,1,\cdot^c,\cup,\cap\}$
(the distance symbol is always implicit, and the measure can be
recovered from it as above).
The class of probability algebras is elementary and admits  a
universal theory denoted $Pr$.
Its model completion is $APr$, the theory of atomless probability
algebras.
It admits quantifier elimination, is $\aleph_0$-categorical (even over
finitely many parameters) and $\aleph_0$-stable (see
\cite{BenYaacov-Usvyatsov:CFO,Berenstein-Henson:ProbabilityWithAutomorphism}).

\begin{dfn}
  Let $\sB$ be a probability algebra.
  An automorphism $\tau \in \Aut(\sB)$ is said to be \emph{aperiodic}
  if for every non-zero event $a \in \sB$ and every $n > 0$
  there is a sub-event $b \subseteq a$ such that
  $\tau^n(b) \neq b$.
  (In other words, the \emph{support} of $\tau^n$ is $1$ for all
  $n \geq 1$.)
\end{dfn}

\begin{fct}[Halmos-Rokhlin-Kakutani Lemma,
  {\cite[386C]{Fremlin:MeasureTheoryVol3}}]
  \label{fct:Rokhlin}
  Let $\sB$ be a probability algebra, $\tau \in \Aut(\sB)$.
  Then $\tau$ is aperiodic if and only if, for every $n \geq 1$
  and every $\varepsilon > 0$ there is $a \in \sB$
  such that $a$, $\tau a$, \ldots, $\tau^{n-1}(a)$ are disjoint
  and $n\mu(a) > 1-\varepsilon$.
\end{fct}

Now let $\cL_\tau = \cL_{Pr} \cup \{ \tau \}$ where
$\tau$ is a new unary function symbol.
Let $APr_\tau$ be the theory
$APr \cup \{\tau \text{ is an automorphism}\}$.
It was shown in \cite{Berenstein-Henson:ProbabilityWithAutomorphism}
that $APr_\tau$ admits a model companion $APr_A$,
consisting of $APr_\tau$ together with axioms saying that $\tau$ is
aperiodic.

\begin{dfn}
  By the \emph{Lebesgue space} we mean the probability space
  $([0,1],\lambda)$, where $\lambda$ is the standard Lebesgue measure.
  The associated probability algebra
  $\sL = \fB([0,1],\lambda)$ is the unique separable
  atomless probability algebra.

  An \emph{automorphism} $\tau$ of the Lebesgue space
  is a measurable, measure-preserving bijection
  between measure one subsets of $[0,1]$.
\end{dfn}

\begin{rmk}
  \begin{enumerate}
  \item 
    Every automorphism of the probability algebra $\sL$
    comes from an automorphism of the Lebesgue space.
  \item
    An automorphism $\tau$ of the Lebesgue space induces an aperiodic
    automorphism on $\sL$ if and only if $\tau$ itself is aperiodic,
    namely if
    $\lambda\{ x\in [0,1]\colon \tau^n(x)=x\} = 0$.
  \end{enumerate}
\end{rmk}

\begin{dfn}
  Let $\sA$ be a probability algebra.
  We equip $\Aut(\sA)$ with the uniform convergence metric
  \begin{gather*}
    d(\tau_0,\tau_1) = \sup_{x \in \sA} d(\tau_0(x),\tau_1(x)).
  \end{gather*}
  Let $(\sA_i,\tau_i) \models APr_A$ for $i = 0,1$
  and let $r \geq 0$.
  Then an \emph{$r$-perturbation} of $(\sA_0,\tau_0)$ to
  $(\sA_1,\tau_1)$ is an (isometric) isomorphism
  $f\colon \sA_0 \cong \sA_1$
  such that
  $d\bigl(  f\tau_0f^{-1}, \tau_1 \bigr) \leq r$.
\end{dfn}

Notice that this is essentially the same definition as for (unit balls
of) Hilbert space.
In particular, as in the Hilbert space case,
this definition satisfies the conditions of
\cite[Theorem~4.4]{BenYaacov:TopometricSpacesAndPerturbations}
and thereby comes from a perturbation system.

\begin{dfn}
  Let $\sA$ be a probability algebra, $\tau \in \Aut(\sA)$,
  $a \in \sA$.
  We say that $(a,\tau)$
  \emph{generate an $(n,\varepsilon)$-partition (of $\sA$)}
  if $a,\tau(a),\ldots,\tau^{n-1}(a)$ are disjoint and
  $\mu\left(\bigvee_{i<n}\tau^i(a)\right) \geq 1-\varepsilon$.
  An $(n,0)$-partition will simply be called an \emph{$n$-partition}.
\end{dfn}

If $\sA_i$, $i = 0,1$ are probability algebras and
$\sB = \sA_0 \otimes \sA_1$ is their free amalgam,
we may identify $\sA_0$ with its image
$\sA_0 \otimes 1 = \{a \otimes 1\colon a \in \sA_0\} \subseteq \sB$,
and similarly $\sA_1 \cong 1 \otimes \sA_1 \subseteq \sB$.
In particular, if $(\sA_i,\tau_i) \models APr_\tau$, $i = 0,1$
then
$(\sA_0 \otimes \sA_1, \tau_0 \otimes \tau_1) \models APr_\tau$ as well.
If in addition $(\sA_0,\tau_0) \models APr_A$
then $\tau_0$ is aperiodic, whereby so
is $\tau_0 \otimes \tau_1$, i.e.,
$(\sA_0 \otimes \sA_1, \tau_0 \otimes \tau_1) \models APr_A$.
Since $APr_A$ is model complete we conclude that
\begin{gather*}
  (\sA_0,\tau_0) \models APr_A, \,
  (\sA_1,\tau_1) \models APr_\tau
  \quad \Longrightarrow \quad
  (\sA_0,\tau_0) \preceq (\sA_0 \otimes \sA_1, \tau_0 \otimes \tau_1).
\end{gather*}

\begin{dfn}
  Let $(\sA,\tau_\sA) \preceq (\sB,\tau_\sB) \models APr_A$ be
  separable.
  We say that $(\sB,\tau_\sB)$ is \emph{partitioned over $\sA$} if:
  \begin{enumerate}
  \item The probability algebra $\sB$ is isomorphic to
    $\sA \otimes \sL$ over $\sA$ (meaning that 
    $a \in \sA$ gets mapped to $a \otimes 1$).
  \item Under this isomorphism,
    for each $0 < n \in \bN$ there exists
    $c_n \in \sL$ such that
    $(1 \otimes c_n,\tau_\sB)$ generate an $n$-partition, all of whose members
    belong to $1 \otimes \sL$.
  \end{enumerate}
\end{dfn}

\begin{lem}
  \label{lem:AutExtPart}
  Let $(\sA,\tau_\sA) \preceq (\sB,\tau_\sB) \models APr_A$ be separable.
  Then there exists a further elementary extension
  $(\sB,\tau_\sB) \preceq (\sC,\tau_\sC)$ which is partitioned over
  $\sA$.
\end{lem}
\begin{proof}
  First of all, we may assume that $\sB$ is atomless over $\sA$.
  Indeed, 
  $(\sB',\tau_\sB') = (\sB,\tau_\sB) \otimes (\sL,\id)$
  is a separable elementary extension of $(\sB,\tau_\sB)$
  and we may replace the latter with the former.
  Therefore we may assume that
  $\sB = \sA \otimes \cL$.

  It is not difficult to construct an automorphism
  $\rho \in \Aut(\sL)$ such that for each $n$ there is
  $c_n \in \sL$ such that $(c_n,\rho)$ generate an $n$-partition.
  Let
  $(\sC,\tau_\sC) = (\sB \otimes \sL, \tau_\sB \otimes \rho)$,
  so $(\sB,\tau_\sB) \preceq (\sC,\tau_\sC)$.

  On the other hand we have
  $\sC = \sB \otimes \sL = \sA \otimes [\sL \otimes \sL]$,
  and for every $k < n$
  \begin{gather*}
    \tau_\sC^k(1 \otimes [1 \otimes c_n])
    = 1 \otimes [1 \otimes \rho^k(c_n)].
  \end{gather*}
  Now use the fact that $\sL \otimes \sL \cong \sL$ to conclude.
\end{proof}

\begin{ntn}
  For a probability algebra $\sC$ and $c \in \sC$, let
  $\sC_{\leq c}$ denote the ideal $\{c' \in \sC\colon c' \leq c\}$.
\end{ntn}

\begin{ntn}
  \label{ntn:SimplePartition}
  For $0 < n \in \bN$
  fix $(\ell_n,\rho_n) \in \sL \times \Aut(\sL)$
  generating an $n$-partition such
  that in addition $(\rho_n)^n = \id$.
  Note that this determines $(\sL,\rho_n)$ up to isomorphism,
  and that it is \emph{not} a model of $APr_A$.
\end{ntn}

\begin{lem}
  \label{lem:FiniteOrderAutPert}
  Let
  $(\sA,\tau_\sA) \preceq (\sB,\tau_\sB) \models APr_A$,
  and assume that $(\sB,\tau_\sB)$ is partitioned over $\sA$.
  Then for every $0 < n \in \bN$ there exists
  $\tau_\sB' \in \Aut(\sB)$ such that:
  \begin{enumerate}
  \item $d(\tau_\sB,\tau_\sB') \leq \frac{1}{2n}$.
  \item
    $(\sB,\tau_\sB') \cong (\sA,\tau_\sA) \otimes (\sL,\rho_n)$
    over $\sA$,
    where $(\cL,\rho_n)$ are as in \fref{ntn:SimplePartition}.
  \end{enumerate}
\end{lem}
\begin{proof}
  We may assume that
  $(\sB,\tau_\sB) = (\sA \otimes \sL,\tau_\sB)$
  and that this identification witnesses that
  $(\sB,\tau_\sB)$ is partitioned over $\sA$.
  Therefore there exists $c \in \sL$ are such that
  $(1 \otimes c,\tau_\sB)$ generate an $n$-partition all
  of whose members are in $1 \otimes \sL$.
  Let
  $1 \otimes c_k = \tau_\sB^k(1 \otimes c) \in 1 \otimes \cL$
  for $k < n$ (or, for that matter, for all $k \in \bZ$).

  Let $\ell_n,\rho_n$ be as in \fref{ntn:SimplePartition}.
  Since $c,\ell_n \in \sL$ and
  $\mu(c) = \mu(\ell_n) = \frac{1}{n}$, there is an isomorphism
  $\theta_0\colon \sL_{\leq c} \cong \sL_{\leq \ell_n}$,
  which induces in turn
  $\theta_1 = \id_{\sA} \otimes \theta_0 \colon
  \sB_{\leq 1 \otimes c}
  \cong \sB_{\leq 1 \otimes \ell_n}$.
  We shall extend $\theta_1$ to an automorphism of $\sB$ as follows.
  For $b \in \sB$, observe that
  $b = \bigvee_{k<n} \bigl( b \wedge (1\otimes c_k) \bigr)$
  is a partition of $b$, and
  $\tau_\sB^{-k}\bigl(b \wedge (1 \otimes c_k) \bigr)
  \in \sB_{\leq 1 \otimes c}$.
  We then define
  \begin{gather*}
    \theta_2(b)
    = \bigvee_{k<n} (\tau_\sA \otimes \rho_n)^k\theta_1\tau_\sB^{-k}
    \bigl(b \wedge (1 \otimes c_k) \bigr).
  \end{gather*}
  For each $k < n$, $\theta_2$ restricts to an isomorphism
  $\sB_{\leq 1 \otimes c_k}
  \cong \sB_{\leq 1 \otimes \rho_n^k(\ell_n)}$,
  so $\theta_2 \in \Aut(\sB)$.
  In addition, if $a \in \sA$ then:
  \begin{align*}
    \theta_2\bigl((a \otimes 1) \wedge (1 \otimes c_k) \bigr)
    &
    = (\tau_\sA \otimes \rho_n)^k\theta_1\tau_\sB^{-k}(a \otimes c_k)
    \\ &
    = (\tau_\sA \otimes \rho_n)^k\theta_1\bigl( \tau_\sA^{-k}(a) \otimes c)
    \\ &
    = (\tau_\sA \otimes \rho_n)^k\bigl( \tau_\sA^{-k}(a) \otimes \ell_n)
    \\ &
    = a \otimes \rho_n^k(\ell_n),
  \end{align*}
  whereby $\theta_2(a \otimes 1) = a \otimes 1$.
  Thus $\theta_2$ acts as the identity on $\sA \otimes 1$.

  Let
  $\tau_\sB' = \theta_2^{-1} (\tau_\sA \otimes \rho_n) \theta_2$.
  For $b \in \sB$ and for $k < n$ let
  $b_k = b \wedge (1 \otimes c_k) \in \sB_{\leq 1 \otimes c_k}$,
  in which case
  $\theta_2(b_k) \in \sB_{\leq 1 \otimes \rho_n^k(\ell_n)}$
  and $(\tau_\sA \otimes \rho_n) \theta_2(b_k) \in
  \sB_{\leq 1 \otimes \rho_n^{k+1}(\ell_n)}$.
  Assume now in addition that $0 \leq k \leq n-2$, i.e., that
  $k+1 \leq n-1$.
  Then in the expression
  $\tau_\sB'(b_k)
  = \theta_2^{-1} (\tau_\sA \otimes \rho_n) \theta_2(b_k)$,
  the instances of $\theta_2^{-1}$ and $\theta_2$ can be expanded
  explicitly as follows:
  \begin{align*}
    \tau_\sB'(b_k) &
    = \theta_2^{-1} (\tau_\sA \otimes \rho_n) \theta_2(b_k)
    \\ &
    =
    \bigl[ \tau_\sB^{k+1} \theta_1^{-1}
    (\tau_\sA \otimes \rho_n)^{-k-1} \bigr]
    (\tau_\sA \otimes \rho_n)
    \bigl[ (\tau_\sA \otimes \rho_n)^k \theta_1 \tau_\sB^{-k} \bigr]
    (b_k)
    \\ &
    = \tau_\sB(b_k).
  \end{align*}
  (On the other hand, this fails for $k = n-1$, since then
  $\theta_2^{-1}$ expands to $\theta_1^{-1}$.)
  It follows that
  $d\bigl( \tau_\sB(b),\tau_\sB'(b) \bigr)
  = d\bigl( \tau_\sB(b_{n-1}),\tau_\sB'(b_{n-1}) \bigr)
  \leq \frac{1}{2n}$
  (since both events are sub-events of same measure
  of $1 \otimes c_{n-1}$ which has measure $\frac{1}{n}$).
  Thus $d(\tau_\sB,\tau_\sB') \leq \frac{1}{2n}$, as desired.
\end{proof}

\begin{thm}
  The theory $APr_A$ is $\aleph_0$-stable up to perturbations of the
  automorphism.
\end{thm}
\begin{proof}
  Let
  $(\sA,\tau_\sA) \models APr_A$.
  By \fref{lem:AutExtPart} $(\sA,\tau_\sA)$ admits a partitioned
  extension $(\sB,\tau_\sB)$.
  We claim that up to perturbation, every type over $\sA$ is realised
  in $(\sB,\tau_\sB)$.

  Indeed any such type is realised in some separable extension
  $(\sC,\tau_\sC) \succeq (\sA,\tau_\sA)$.
  By \fref{lem:AutExtPart} again we may assume that
  $(\sC,\tau_\sC)$ is also partitioned over $\sA$.
  By \fref{lem:FiniteOrderAutPert} any two partitioned extensions of
  $\sA$ admit an $\frac{1}{2n}$ perturbation over $\sA$ to
  $(\sA,\tau_\sA) \otimes (\sL,\rho_n)$, and composing these
  we obtain an $\frac{1}{n}$-perturbation
  $(\sC,\tau_\sC) \to (\sB,\tau_\sB)$ which fixes $\sA$,
  and this for every $0 < n \in \bN$.
\end{proof}

\section{Non superstability of probability algebras with a generic
  automorphism}
\label{sec:ProbAutNotSuperstable}

The second author and Henson
\cite{Berenstein-Henson:ProbabilityWithAutomorphism}
asked whether the theory of probability algebras with a generic
automorphism is superstable, as one might expect by analogy with a
theorem of Chatzidakis and Pillay
\cite{Chatzidakis-Pillay:GenericStructuresAndSimpleTheories}
regarding generic automorphisms in classical first order
logic.
In this section we present the first author's negative answer,
which chronologically came before the results in
\fref{sec:Probability}, and to a large extent motivated them.

Our aim is to show that the theory $APr_A$ admits many types over small
sets of parameters, and for this purpose it will suffice to show that
there are many $1$-types over parameters which belong to the fixed
algebra of the automorphism.
We therefore proceed in two steps, first characterising such types and
then showing there are many of them.
Throughout we let $(\cU,\tau)$ denote a monster model of $APr_A$,
and let $\cU^\tau$ denote the fixed sub-algebra of $\cU$ under
$\tau$.

\subsection{Types over the empty set and over fixed sub-algebras}

Let us try to describe the space of $1$-types in $APr_A$ over a set
of parameters contained in $\cU^\tau$.
We start with types over the empty set.

Consider a function $\eta\colon 2^{<\omega} \to [0,1]$
sending $s \mapsto \eta_s$.
We call such a function \emph{shift invariant} if
\begin{gather*}
  \tag{SI}\label{eq:ShiftInvariant}
  \eta_\emptyset = 1,
  \qquad
  \eta_{s\concat 0} + \eta_{s\concat 1}
  = \eta_{0\concat s}+ \eta_{1\concat s}
  = \eta_s, \qquad \text{for all } s \in 2^{<\omega}.
\end{gather*}
We define $X \subseteq [0,1]^{2^{<\omega}}$ to consist of all shift
invariant mappings.
This is a closed subset of $[0,1]^{2^{<\omega}}$, and therefore
compact.

Let $p \in \tS_1(APr_A)$, and
for $n \in \bN$, $s \in 2^n$, let
$\eta_{p,s} = \mu\left( \bigwedge_{i<n} \tau^i(x^{s_i}) \right)^p$
(where $x^0 = x$, $x^1 = x^c$ is the natural action of $(\bZ/2\bZ,+)$).
Then $\eta_p\colon s \mapsto \eta_{p,s}$ is shift invariant,
yielding a mapping $\rho\colon S_1(APr_A) \to X$
sending $p \mapsto \eta_p$.
This mapping $\rho$ is clearly continuous, and by
quantifier elimination it is injective.

Conversely, let $\eta \in X$,
and let $\sA$ be any sufficiently homogeneous atomless
probability algebra.
Then one can find in
$\sA$ a sequence of events $(a_n)$
such that
$\mu(a_0) = \eta_0$, $\mu(a_0)^c = \eta_1 = 1-\eta_0$,
and in general, for every $n$ and $s \in 2^n$:
\begin{gather*}
  \mu\Bigl( a_n \wedge \bigwedge_{k<n} a_k^{s_k} \Bigr)
  = \eta_{s\concat 0},
  \qquad
  \mu\Bigl( a_n^c \wedge \bigwedge_{k<n} a_k^{s_k} \Bigr)
  = \eta_{s\concat 1} = \eta_s - \eta_{s\concat 0}.
\end{gather*}
This is indeed consistent by shift invariance.
Moreover, shift invariance implies that for every
$n,k \in \bN$ and $s\in 2^n$:
$\eta_s = \mu\left( \bigwedge_{i<n} a_{k+i}^{s_i} \right)$
(by induction on $k$).
It follows by quantifier elimination that the mapping
$a_n \mapsto a_{n+1}$ is elementary and therefore extends to an
automorphism $\tau_\sA \in \Aut(\sA)$,
and we may embed $(\sA,\tau_\sA)$ in $(\cU,\tau)$.
In other words, for every $\eta \in X$ we can find
$a \in \cU$ such that
$\eta_s = \mu\left( \bigwedge_{i<n} \tau^i(a^{s_i}) \right)$
for all $s \in 2^{<\omega}$.
Thus $\eta = \rho(\tp(a))$, showing that $\rho$ is bijective.
Since it is also continuous, from a compact to a Hausdorff space, it
is a homeomorphism.

If $Y$ is an arbitrary topological space
we have
$C(Y,[0,1])^{2^{<\omega}} = C\bigl( Y, [0,1]^{2^{<\omega}} \bigr)$
as sets.
Equipping $C(Y,Z)$ with the compact-open topology
and $2^{<\omega}$ with the discrete topology these are homeomorphisms.
(The common topology can be given by a sub-basis, where a sub-basic
open set is of the form
$\bigl\{ f \in C(Y\times 2^{<\omega},[0,1])
\colon f[K\times\{s\}] \subseteq U\bigr\}$,
with $K \subseteq Y$ compact, $s \in 2^{<\omega}$ and
$U \subseteq [0,1]$ is open).
We may define when a mapping
$\eta \in C(Y,[0,1])^{2^{<\omega}}$
is shift invariant by \fref{eq:ShiftInvariant} as above,
and let $X_Y$ be the set of all such shift invariant functions.
It is then clear that $X_Y = C(Y,X)$.

We now turn to types over a sub-algebra $\sA \subseteq \cU^\tau$,
namely over parameters fixed by $\tau$.
We shall use the following.
\begin{fct}
  Let $\sA$ be a probability algebra, and let
  $\widetilde \sA$ be the Stone space of the underlying Boolean algebra.
  For an event $a \in \sA$ let $\tilde a \subseteq \widetilde \sA$ be
  the corresponding clopen set.
  \begin{enumerate}
  \item
    The space $\widetilde \sA$ admits a unique regular Borel
    probability measure $\tilde \mu$ such that
    $\tilde \mu(\tilde a) = \mu(a)$
    for all $a \in \sA$.
  \item
    The probability algebra of $(\widetilde \sA,\tilde \mu)$
    is canonically isomorphic to
    $\sA$, identifying the equivalence class of $\tilde a$ with the
    event $a$.
  \item
    The natural mapping
    $C(\widetilde \sA,\bR) \to L_\infty(\widetilde \sA,\tilde \mu)$
    is bijective.
    In other words, every equivalence class of bounded Borel functions
    up to equality $\tilde \mu$-a.e.\ contains a unique continuous
    representative.
  \item
    Let $\sB \supseteq \sA$ be a larger probability algebra.
    Then there exists a conditional expectation operation
    $\bE[\cdot|\sA]\colon L_1(\widetilde \sB) \to L_1(\widetilde \sA)$
    where $\bE[f|\sA]$ is the unique function such that
    for all $a \in \sA$:
    \begin{gather*}
      \int_{\tilde a \subseteq \widetilde \sB} f\, d\tilde \mu_\sB
      =
      \int_{\tilde a \subseteq \widetilde \sA} \bE[f|\sA]\,d\tilde \mu_\sA.
    \end{gather*}
  \end{enumerate}
\end{fct}
\begin{proof}
  See \cite{Fremlin:MeasureTheoryVol3}.
  The Stone space associated to a Boolean algebra is discussed in 311.
  The construction of the measure $\tilde \mu$ and the isomorphism
  between $\sA$ and the probability algebra of
  $(\widetilde \sA,\tilde \mu)$ appears in 321J.
  The identification of $C(\widetilde \sA,\bR)$ and
  $L_\infty(\widetilde \sA,\tilde \mu)$ can be found in
  363 and 364K.
  Finally, conditional expectations are discussed in 365R.
\end{proof}

Now let $\sA \subseteq \cU^\tau \subseteq \cU$,
and let $p \in \tS_1(\sA)$, say realised by $b \in \cU$.
For $s \in 2^n$, the type $p$ determines the mapping associating
to every $a \in \sA$
the measure
$\mu\left( a \wedge \bigwedge_{i<n} \tau^i(b^{s_i}) \right)$.
In other words, $p$ determines the function
$\eta_{p,s}
= \bP\Bigl[ \bigwedge_{i<n} \tau^i(b^{s_i}) \Big| \sA \Bigr]
\in L_1(\widetilde \sA)$.
Since the essential range of $\eta_{p,s}$ lies in $[0,1]$ we have
$\eta_{p,s} \in L_\infty(\widetilde \sA) = C(\widetilde \sA,\bR)$,
and in fact $\eta_{p,s} \in C(\widetilde \sA,[0,1])$.
Let
$\eta_p \in C(\widetilde \sA,[0,1])^{2^{<\omega}}$,
$\eta_p\colon s \mapsto \eta_{p,s}$.
It is not difficult to see that
$\eta_p$ satisfies \fref{eq:ShiftInvariant}, so
identifying
$C(\widetilde \sA,[0,1])^{2^{<\omega}}$
with
$C(\widetilde \sA,[0,1]^{2^{<\omega}})$
we actually have
$\eta_p \in X_\sA := C(\tilde \sA,X)$.
We have thus obtained a mapping
$\rho_\sA \colon \tS_1(\sA) \to X_\sA$.
Again, it is injective by quantifier elimination and continuous,
and a construction as above yields that it is
surjective.
We have thus obtained a homeomorphism
\begin{gather*}
  \rho_\sA \colon \tS_1(\sA) \cong X_\sA = C(\widetilde \sA,X).
\end{gather*}

For a closed set $K \subseteq X$ and $a \in \sA$ define
\begin{gather*}
  K^a
  = \{ \eta \in X_\sA
  \colon
  \eta[\tilde a] \subseteq K \}.
\end{gather*}
It is not difficult to see that $K^a$ is closed in the compact-open
topology.
If $\pi(x)$ is a partial type over $\emptyset$ and
$K \subseteq X$ corresponds to the closed set
$[\pi] \subseteq \tS_1(APr_A)$,
then the closed set $K^a$ corresponds to
$[\pi^a] \subseteq \tS_1(\sA)$,
where $\pi^a$ is a partial type over $\sA$.
If $\pi^a(b)$ holds we say that $b$ \emph{satisfies $\pi$ over $a$}.

Let $d_\sA$ denote the distance function between types over $\sA$, and
similarly $d_\emptyset$.
It is fairly easy to verify that the distance mapping
$d_\emptyset\colon \tS_1(APr_A)^2 \to [0,1]$,
i.e., $d_\emptyset\colon X^2 \to [0,1]$,
is Borel measurable
(but not continuous, since the theory is not $\aleph_0$-categorical).
Thus, if $p,q \in \tS_1(\sA)$,
then $d_\emptyset \circ (\eta_p,\eta_q)$ is a random
variable from $\sA$ to $[0,1]$, which we can integrate.

\begin{lem}
  For all $p,q \in \tS_1(\sA)$: $d_\sA(p,q) \geq \int d_\emptyset\circ(\eta_p,\eta_q)\, d\tilde \mu$.
\end{lem}
\begin{proof}
  Assume $a \models p$, $b \models q$.
  Let $g = \bP[a\triangle b|\sA]$.
  Then $g \geq d_\emptyset\circ(\eta_p,\eta_q)$.
\end{proof}

\subsection{Non superstability proof}

We say that a continuous theory is \emph{small} if the metric on
$\tS_n(APr)$ is separable for all $n \in \bN$.
\begin{lem}
  The theory $APr_A$ is not small.
  More precisely, there is an uncountable family of types over the
  empty set every two of which have distance $\geq\frac{1}{2}$.
\end{lem}
\begin{proof}
   For every real $\alpha$, let $p_\alpha$ be the type of one half of the
  circle on which $\tau$ acts by rotation by $2\pi\alpha$  (the measure on the
  circle being the Lebesgue measure normalised to have total length
  $1$).
  If $\alpha,\beta \geq 0$ are irrational and linearly independent
  over the rationals then for every $\varepsilon > 0$ there exist
  $n,k,\ell \in \bN$ such that
  $|n\alpha-k|, |n\beta-\ell-\half| < \varepsilon$.
  If $a \models p_\alpha$, $b \models p_\beta$ then 
  $d(a,\tau^n(a)) = \mu(a\triangle \tau^n(a)) < 2\varepsilon$
  while
  $d(b,\tau^n(b)) > 1 - 2\varepsilon$.
  It follows that
  $2d(a,b) = d(a,b)+d(\tau^n(a),\tau^n(b)) \geq 1-4\varepsilon$,
  namely
  $d(a,b) \geq \frac{1}{2} - \varepsilon$.
  Therefore $d(p_\alpha,p_\beta) \geq \frac{1}{2}$
  (it is not difficult to check that the distance between the types is
  in fact equal to $\half$).

  Let $S \subseteq \bR$ be a vector space base for $\bR$ over $\bQ$.
  Then follows that $\{p_\alpha\colon \alpha \in S\}$ is a continuum-size set of
  equally distanced types.
\end{proof}

\begin{prp}
  \label{prp:TypeTree}
  There exists a family $\{\pi_s(x)\}_{s \in 2^{<\omega}}$
  of partial types over $\emptyset$ such that
  \begin{enumerate}
  \item
    For all $s \in 2^{<\omega}$ we have
    $d(\pi_{s \concat 0},\pi_{s \concat 1}) \geq \frac{1}{3}$,
    meaning that $d(a_0,a_i) \geq \frac{1}{3}$
    whenever $a_i \models \pi_{s \concat i}$.
  \item If $s, t \in 2^{<\omega}$ and $t$ extends $s$ then 
    $\pi_t \vdash \pi_s$.
  \end{enumerate}
\end{prp}
\begin{proof}
  This is a metric Cantor-Bendixson rank argument which applies more
  generally, saying that if $T$ is a non-small theory with a countable
  language then such a tree exists (with $\frac{1}{3}$ possibly
  replaced with another positive constant).
  For an even more general statement of this fact see
  \cite[Propositions 3.16 and 3.19]{BenYaacov:TopometricSpacesAndPerturbations}.

  Define compacts subsets
  $X_\alpha \subseteq \tS_1(\emptyset)$ by induction on $\alpha$.
  Start with $X_0 = \tS_1(\emptyset)$;
  for $\alpha$ limit, $X_\alpha = \bigcap_{\beta<\alpha} X_\beta$;
  and given $X_\alpha$, obtain $X_{\alpha+1}$ by removing from $X_\alpha$ all
  points for which there is a relatively open neighbourhood of
  diameter $< \frac{1}{2}$.

  Since the language is countable, the topology on $X$ admits a
  countable base.
  If we only take out basic open sets of diameter $< \frac{1}{2}$ we
  still get the same sequence $X_\alpha$, and since the base is countable
  the sequence stabilises before $\aleph_1$.
  Let $S \subseteq \tS_1(\emptyset)$ be an uncountable subset of types every pair of
  which have distance $\geq \frac{1}{2}$.
  Each set of diameter $< \frac{1}{2}$ can contain at most one member
  of $S$, so $X_{\aleph_1} \neq \emptyset$.
  The topological space $X_{\aleph_1}$ (with the induced topology) is
  $\frac{1}{2}$-perfect, meaning that every non-empty open subset has
  diameter $\geq \frac{1}{2}$.

  Let
  $D = \{(q,r) \in \tS_1(APr_A)^2\colon d(q,r) \leq \frac{1}{3}\}$
  and $D_1 = D \cap X_{\aleph_1}^2$.
  Then $D$ is closed, being the image of the closed
  set $[d(x,y) \leq \frac{1}{3}] \subseteq \tS_2(APr_A)$
  under the projection
  $\tS_2(APr_A) \to \tS_1(APr_A)^2$ (a continuous mapping
  from a compact space to a Hausdorff space is always closed).
  Thus $D_1$ is closed in $X_{\aleph_1}^2$.

  Through the end of the proof we work in $X_{\aleph_1}$, with the induced
  topology.
  In particular, if $Y \subseteq X_{\aleph_1}$ then $Y^\circ$ denotes
  the interior of $Y$ in this topology.

  We start with $\pi_\emptyset(x)$ being the partial type defining
  $X_{\aleph_1}$.
  It has the property that
  $[\pi_\emptyset]^\circ \neq \emptyset$.
  Assume now we have $\pi_s$ such that
  $[\pi_s]^\circ \neq \emptyset$.
  The interior has diameter $\geq \frac{1}{2}$, so there are
  $q,r \in [\pi_s]^\circ$ such that $d(r,q) > \frac{1}{3}$.
  Thus $(q,r) \notin D_1$, so they admit open neighbourhoods
  $q \in U \subseteq [\pi_s]^\circ$ and
  $r \in V \subseteq [\pi_s]^\circ$ such that
  $(U \times V) \cap D_1 = \emptyset$.
  We can then find smaller open neighbourhoods such that
  $q \in U_1 \subseteq \overline U_1 \subseteq U$ and
  $r \in V_1 \subseteq \overline V_1 \subseteq V$.
  Letting $\pi_{s\concat 0}$ be the partial type defining
  $\overline U_1$ and $\pi_{s\concat 1}$ the partial type defining
  $\overline V_1$
  we get:
  $[\pi_{s}]
  \supseteq [\pi_{s\concat 0}]
  \supseteq [\pi_{s\concat 0}]^\circ
  \neq \emptyset$
  and 
  $[\pi_{s}]
  \supseteq [\pi_{s\concat 1}]
  \supseteq [\pi_{s\concat 1}]^\circ
  \neq \emptyset$.
  Finally,
  $([\pi_{s\concat 0}] \times [\pi_{s\concat 1}]) \cap D_1
  = \emptyset$
  implies
  $d(\pi_{s\concat 0},\pi_{s\concat 1}) > \frac{1}{3}$.

  Repeating this argument we obtain the required partial types.
\end{proof}

\begin{lem}
  The theory $APr_A$ $\lambda$-stable if and only if
  $\lambda^{\aleph_0} = \lambda$.
\end{lem}
\begin{proof}
  One direction is since $APr_A$ is stable in a countable language.

  For the other, assume $\lambda^{\aleph_0} > \lambda$.
  Let $\{\pi_s\colon s \in 2^{<\omega}\}$ be as in \fref{prp:TypeTree}.
  Let $\{a_i\colon i < \lambda\}$ be a sequence of independent events of
  measure $\frac{1}{2}$, all fixed by $\tau$, and let $\sA$ be
  the generated complete algebra.

  For $\theta \in \lambda^\bN$ and $s \in 2^{<\omega}$ let
  $b_{\theta,s} = \bigwedge_{i<|s|} a_{\theta(i)}^{s(i)} \in \sA$,
  and let
  \begin{gather*}
    \rho_{\theta,n} = \bigcup_{s\in 2^n} \pi_s^{b_{\theta,s}}(x),
    \qquad
    \rho_\theta = \bigcup_n \rho_{\theta,n}.
  \end{gather*}
  In other words,
  $\rho_\theta(c)$ holds if and only if, for every
  $s \in 2^{<\omega}$,
  $c$ satisfies $\pi_s$ over
  $\bigwedge_{i<|s|} a_{\theta(i)}^{s(i)}$.
  It is easy to check that $\rho_{\theta,n}$ is consistent and implies
  $\rho_{\theta,m}$ for $m < n$,
  so $\rho_\theta$ is consistent as well.
  Choose for each $\theta$ a complete type $r_\theta \in \tS_1(\sA)$ extending
  $\rho_\theta$.

  Let $\theta \neq \theta' \in \lambda^\bN$, and let $i \in \bN$ be
  such that $\theta(i) \neq \theta'(i)$.
  Then over
  $a_{\theta(i)} \setminus a_{\theta'(i)}$, $\eta_{r_\theta}$
  takes only values in
  $\bigcup_{s \in 2^i} [\pi_{s\concat 0}] \subseteq \tS_1(APr_A)$,
  while $\eta_{r_{\theta'}}$ only takes values in
  $\bigcup_{s \in 2^i} [\pi_{s\concat 1}]$,
  and the opposite holds over
  $a_{\theta'(i)} \setminus a_{\theta(i)}$.
  Thus
  $d_\emptyset\circ(\eta_{r_\theta},\eta_{r_{\theta'}})
  \geq \frac{1}{3}$
  over
  $a_{\theta(i)}\triangle a_{\theta'(i)}$, which has measure $\half$.
  We conclude that $d(r_\theta,r_{\theta'}) \geq \frac{1}{6}$.

  We have shown that there are $\lambda^\bN$ equally distanced types
  over a set of $\lambda$ parameters, as desired.
\end{proof}

We conclude:
\begin{thm}
  The theory $APr_A$ is not superstable, and therefore not supersimple.
\end{thm}

Notice the difference from the case of $IHS_A$,
which is superstable (but not $\aleph_0$-stable).

\providecommand{\bysame}{\leavevmode\hbox to3em{\hrulefill}\thinspace}
\providecommand{\MR}{\relax\ifhmode\unskip\space\fi MR }
\providecommand{\MRhref}[2]{%
  \href{http://www.ams.org/mathscinet-getitem?mr=#1}{#2}
}
\providecommand{\href}[2]{#2}

\end{document}